\documentclass[reqno, 11pt]{amsart}
\usepackage{amsmath}	
\usepackage{amssymb}	
\usepackage{amsthm}
\usepackage{amsfonts}
\usepackage{latexsym}
\usepackage{mathrsfs}
\usepackage[all,ps,cmtip]{xy}
\usepackage{tabularx}
\usepackage{pict2e}
\usepackage{graphicx}
\usepackage{xypic}
\usepackage{hyperref}
\usepackage{tikz}
\usepackage{makecell}
\usepackage{diagbox}

\usepackage[text={155mm, 235mm},centering]{geometry}
\input diagxy
\input xy
\theoremstyle{plain}
\newtheorem{theorem}{Theorem}[section]
\newtheorem{lemma}[theorem]{Lemma}
\newtheorem{proposition}[theorem]{Proposition}
\newtheorem{corollary}[theorem]{Corollary}

\theoremstyle{definition}
\newtheorem{definition}[theorem]{Definition}
\newtheorem{remark}[theorem]{Remark}
\newtheorem{example}[theorem]{Example}
\newtheorem{quest}[theorem]{Question}

\DeclareMathOperator{\A}{\mathcal{A}}
\DeclareMathOperator{\CO}{\mathcal{O}}

\allowdisplaybreaks

\numberwithin{equation}{section}

\begin{document}

\title[]{Formality of the Dolbeault complex and deformations of holomorphic Poisson manifolds}

\author[Y. Chen]{Youming Chen }
\address{School of Science, Chongqing University of Technology, Chongqing 400054, PR China}%
\email{youmingchen@cqut.edu.cn}%


\begin{abstract}
The purpose of this paper is to study the properties of
holomorphic Poisson manifolds $(M,\pi)$
under the assumption of $\partial_{}\bar{\partial}$--lemma or
$\partial_{\pi}\bar{\partial}$--lemma.
Under these assumptions,
we show that the Koszul--Brylinski homology can be recovered by
the Dolbeault cohomology,
and prove that the DGLA
$(A_M^{\bullet,\bullet},\bar{\partial},[-,-]_{\partial_\pi})$ is formal.
Furthermore,
we discuss the Maurer--Cartan elements of $(A_M^{\bullet,\bullet}[[t]],\bar{\partial},[-,-]_{\partial_\pi})$
which induce the deformations of complex structure of $M$.
\end{abstract}

\maketitle

\setcounter{tocdepth}{1}


\section{Introduction}
In the past decades,
as important properties,
the $\partial_{}\bar{\partial}$--lemma and its variation has been extensively studied
in complex geometry, symplectic geometry and generalized complex geometry;
see \cite{DGMS75,BK98,Merkulov98,Man99,Cav06,TY12-1,TY12-2, AT13, AT15,TY16,YY20} etc.
Among these works,
a profound one of
Deligne, Griffiths, Morgan, and Sullivan
is that if a compact complex manifold satisfies the $\partial_{}\bar{\partial}$--lemma,
then it is formal.
In symplectic geometry,
a variation of $\partial_{}\bar{\partial}$--lemma,
called $d\delta$--lemma,
also attracts many attention:
for a compact symplectic manifold,
it satisfies the $d\delta$--lemma if and only if the Hard Lefschetz condition holds,
and in this case the symplectic manifold is formal,
and its de Rham cohomology admits a structure of Frobenius manifold.

In this paper we consider the holomorphic Poisson manifold.
Originally, Poisson structures arise from Hamiltonian system of classical dynamics.
In many situations, the Poisson structures are actually
holomorphic; see \cite{BZ99,Hit06,L-GSX08,Go10,Hit12,CFP16,CGP15} etc.
In particular, holomorphic Poisson structures are closely related to generalized complex geometry.
On the one hand, a holomorphic Poisson structure naturally defines a generalized
complex structure of special type; see \cite{Gu11}.
On the other hand, by Bailey's local classification theorem of generalized complex structures,
each generalized complex manifold is locally equivalent to the product of a symplectic manifold
and a holomorphic Poisson manifold; see \cite{Ba13}.
We refer the readers to \cite{Hit03,Hit06,Fu05,L-GSX08,Go10,CSX10,Gu11,BX15,BCV19}
and references therein for more results on the applications of holomorphic
Poisson structures
in generalized complex geometry and the relationships with other geometries.

Assume that $(M,\pi)$ is a compact holomorphic Poisson manifold.
The Koszul--Brylinski operator $\partial_{\pi}:=\iota_{\pi}\partial-\partial\iota_{\pi}$,
studied independently by Koszul \cite{Kos84} and Brylinski \cite{Bry88},
has many properties.
The $k$--th Koszul--Brylinski homology $H_{k}(M, \pi)$ of $(M,\pi)$
is defined as the $k$-th hypercohomology of holomorphic Koszul--Brylinski complex
$$
  \cdots \stackrel{\partial_{\pi}}{\to}\Omega_{M}^{p+1} \stackrel{\partial_{\pi}}{\to}
      \Omega_{M}^{p} \stackrel{\partial_{\pi}}{\to} \Omega_{M}^{p-1} \stackrel{\partial_{\pi}}{\to} \cdots.
$$
Moreover, as a BV operator,
$\partial_{\pi}$ generates a Lie bracket $[-,-]_{\partial_\pi}$ (defined by \ref{Lie str on forms1}).
Most notably, there exists a holomorphic version of Evens--Lu--Weinstein duality
for $H_{\bullet}(X,\pi)$,
which is  a generalization of Serre duality for Dolbeault cohomology; see \cite[Theorem 4.4]{Sti11}.
Furthermore, there is a canonical Fr\"{o}hlicher--type spectral sequence, called the
Dolbeault--Koszul--Brylinski spectral sequence (see Definition \ref{Dol-Poi-spectral-seq}),
which converges to $H_{\bullet}(X,\pi)$.
However, in general,
it is not easy to compute the holomorphic
Koszul--Brylinski homology for a specific holomorphic Poisson manifold.
As far as we know, only some particular class are calculated(see \cite{Hon19,HX11,Sti11}).
Essentially, the Koszul--Brylinski homology and the algebraic structures on it
naturally depend both on the Poisson structures
and complex structures, but the Poisson structures are some kind complicated.
Thus, a natural question, motivated by Brylinski \cite{Bry88}, arises now is:

\begin{quest}\label{pro}
What conditions on a holomorphic Poisson manifold $(M,\pi)$
ensure the degeneracy of $E_1$--page of the
Dolbeault--Koszul--Brylinski spectral sequence?
\end{quest}

Motivated by the works of $\partial_{}\bar{\partial}$--lemma above
and the close relations between holomorphic
Poisson geometry and symplectic geometry, generalized complex geometry,
we consider the $\partial_{}\bar{\partial}$--lemma and $\partial_{\pi}\bar{\partial}$--lemma
on holomorphic Poisson manifold.
Our first main result states as follows.
\begin{theorem}\label{main-theorem1}
Let $(M,\pi)$ be a holomorphic Poisson manifold.
If $M$ satisfies the $\partial_{}\bar{\partial}$--lemma or $\partial_{\pi}\bar{\partial}$--lemma,
then its Dolbeault--Koszul--Brylinski spectral sequence degenerates at $E_{1}$--page:
$$
H_{k}(M, \pi)
\cong
\bigoplus_{p-q=n-k}H^{p,q}_{\bar{\partial}}(M).
$$
\end{theorem}
Notice that if $M$ is a projective manifold or K\"{a}hler manifold,
then it automatically satisfies the $\partial\bar{\partial}$--lemma,
and hence the Theorem \ref{main-theorem1} is applicable to these situations.

The Lie bracket $[-,-]_{\partial_\pi}$ on the Dolbeault complex of $M$ which is generated by the
Koszul--Brylinski operator $\partial_{\pi}$,
is in fact compatible
with the Dolbeault operator,
i.e., the Dolbeault complex of $M$ admits a DGLA structure.
Analogous to the symplectic case,
we consider the formality of such DGLA.
\begin{theorem}\label{main-theorem2}
Suppose that $(M,\pi)$ is a holomorphic Poisson manifold.
If $(M,\pi)$ satisfies the $\partial_{\pi}\bar{\partial}$--lemma,
then the DGLA
$(A_M^{\bullet,\bullet},\bar{\partial},[-,-]_{\partial_\pi})$ is formal.
\end{theorem}

Actually,
the DGLA $(A_{M}^{\bullet,\bullet},\bar{\partial},[-,-]_{\partial_\pi})$
is closely relevant with the Kodaira--Spencer DGLA
$(A^{0,\bullet}(M,\wedge^\bullet T_{M}),\bar{\partial},[-,-]_{})$
which controls the deformations of complex structure of $M$;
see Proposition \ref{deformation}.
Therefore it is worth seeking the
Maurer--Cartan elements of the DGLA $(A_M^{\bullet,\bullet},\bar{\partial},[-,-]_{\partial_\pi})$.
\begin{theorem}\label{main-theorem3}
Let $(M,\pi)$ be a holomorphic Poisson manifold.
If $(M,\pi)$ satisfies the $\partial_{}\bar{\partial}$--lemma
or $\partial_{\pi}\bar{\partial}$--lemma,
then for any $[\alpha]\in H_{\bar{\partial}}^{1,1}(M)$,
there exists an Maurer--Cartan element
$\alpha_t$
whose $\alpha_1$ is a representative of $[\alpha]$.
In this case, $\pi^\sharp[\alpha]\in H^1(M,\mathcal{T}_{M})$
is tangent to a deformation of complex structure.
\end{theorem}
It is noteworthy that
the fact that on a holomorphic Poisson manifold $(X,\pi)$,
for any $[\alpha]\in H_{\bar{\partial}}^{1,1}(M)$,
$\pi^\sharp[\alpha]\in H^1(M,\mathcal{T}_{M})$
is tangent to a deformation of complex structure,
is proved by Hitchin \cite{Hit12}
under the assumption of $\partial_{}\bar{\partial}$--lemma or $H^2(M,\CO_M)=0$.
Later,
Hitchin's result is generalized by Fiorenza and Manetti (cf. \cite[Theorem 6.3]{FM12}) under the assumption that the natural map
$H^2_{dR}(M,\mathbb{C})\to H^2(M,\CO_M)$ is surjective.

This paper is organized as follows.
In Section [\ref{c2}] we first review some basics of holomorphic Poisson manifold
and discuss some results on $\partial_{}\bar{\partial}$--lemma and $\partial_{\pi}\bar{\partial}$--lemma.
We devote Section [\ref{c3}] to the formality of the Dolbeault complex.
In Section [\ref{c4}] we study the Maurer--Cartan equation of the DGLA
on the Dolbeault complex.
Finally, some examples are stated in Section [\ref{c6}].

\subsection*{Acknowledgments}
The author would like to thank the School of Mathematics of Sichuan University
and Tianyuan Mathematical Center in Southwest China for the hospitalities during the spring of 2022.
In particular, the author would like to thank the referee
for introducing the example of Nakamura manifold in the subsection [\ref{Nak}] to him.
This work is partially supported by the National Nature Science Foundation of China
(No. 12126309, 12126354, 12171351),
the Natural Science Foundation of Chongqing (No. CSTC2020JCYJ--MSXMX0160),
and the Scientific and Technological Research Program of Chongqing Municipal
Education Commission (Grant No. KJQN202201108).

\section{Preliminaries}\label{c2}

\subsection{Basics of holomorphic Poisson manifold}

In this subsection, we recall some basic facts on holomorphic Poisson manifolds.
Let $M$ be a complex manifold
and let $\CO_M$ be its structure sheaf (i.e., the sheaf of holomorphic functions),
$\Omega_{M}^{p}$ be the sheaf of holomorphic $p$--forms,
$\mathcal{T}_{M}$ be the sheaf of holomorphic vector fields.
\begin{definition}
A complex manifold $M$ is called a {\it holomorphic Poisson manifold} if
$M$ admits a section $\pi\in \Gamma(M, \wedge^{2}\mathcal{T}_{M})$
with $\bar{\partial}\pi=0, [\pi,\pi]_{SN}=0$
where $[-,-]_{SN}$ means the Schouten bracket.
\end{definition}
Such $\pi$ is called a {\it holomorphic Poisson bi--vector field} of $M$.
It induces a sheaf morphism $\pi^\sharp:\Omega_{M}^{1}\to\mathcal{T}_{M}$ via
$$\pi^\sharp(df)(dg)=\pi(df,dg).$$
The {\it Koszul--Brylinski operator} of $(M, \pi)$
on the sheaf $\Omega_{M}^{\bullet}$ of holomorphic forms
of $M$ is defined as
$$
\partial_{\pi}:= \iota_{\pi}\circ \partial-\partial\circ \iota_{\pi},
$$
where $\partial$ is the Dolbeault operator and $\iota_{\pi}$ is the contraction operator
with respect to $\pi$.
One can verify the following lemma by direct calculations.
\begin{lemma}\label{identities}
Let $(M, \pi)$ be a holomorphic Poisson manifold. Then we have the following identities:
\begin{itemize}
\item[(1)]$\bar{\partial}\iota_{\pi}-\iota_{\pi}\bar{\partial}=0,$
$\bar{\partial}\partial_{\pi}+\partial_{\pi}\bar{\partial}=0;$
\item[(2)]$\partial\partial_{\pi}+\partial_{\pi}\partial=0;$
\item[(3)]$\partial_{\pi}\iota_{\pi}-\iota_{\pi}\partial_{\pi}=0,$
$\partial_{\pi}^2=0$.
\end{itemize}
\end{lemma}

Moreover, one can check that
the Koszul--Brylinski operator $\partial_{\pi}$ is a BV operator, i.e.,
for any  $\alpha\in \Omega_{M}^{k}$ and $\beta\in \Omega_{M}^{l}$,
\begin{equation}\label{Lie str on forms1}
[\alpha, \beta]_{\partial_{\pi}}
=(-1)^{k}\Big(\partial_{\pi}(\alpha\wedge \beta)-(\partial_{\pi}\alpha) \wedge \beta
               -(-1)^{k} \alpha\wedge (\partial_{\pi}\beta)\Big)
\end{equation}
is a Gerstenhaber bracket (a Poisson bracket with degree $-1$) on $\Omega_{M}^{\bullet}$.
Equivalently, the bracket $[-, -]_{\partial_{\pi}}$ is obtained by Leibniz rule via
\begin{equation}\label{Lie str on forms2}
[\alpha, \beta]_{\partial_{\pi}}
:=L_{\pi^{\sharp}(\alpha)}\beta-L_{\pi^{\sharp}(\beta)}\alpha-\partial(\pi(\alpha,\beta)),
\; \forall \; \alpha,\beta\in \Omega_{M}^{1}.
\end{equation}
Thus, there is a  holomorphic Koszul--Brylinski complex
$$
0\to \Omega_{M}^{n} \stackrel{\partial_{\pi}}{\to}
  \cdots \stackrel{\partial_{\pi}}{\to}\Omega_{M}^{s+1} \stackrel{\partial_{\pi}}{\to}
      \Omega_{M}^{s} \stackrel{\partial_{\pi}}{\to} \Omega_{M}^{s-1} \stackrel{\partial_{\pi}}{\to} \cdots\stackrel{\partial_{\pi}}{\to}\mathcal{O}_M\to 0.
$$

\begin{definition}
Let $(M, \pi)$ be a holomorphic Poisson manifold.
The {\it $k$--th Koszul--Brylinski homology} of $(M,\pi)$ is defined as
the $k$--th hypercohomology of its holomorphic Koszul--Brylinski complex, that is to say,
\begin{equation*}
H_{k}(M, \pi):=\mathbb{H}^{k}(M, (\Omega_{M}^{\bullet},\partial_{\pi})).
\end{equation*}
\end{definition}

\begin{lemma}\label{KB-equal1}
Suppose $(M, \pi)$ is a holomorphic Poisson manifold.
Then its holomorphic Koszul--Brylinski complex admits a resolution which is
the total complex of the double complex
$(\A_{M}^{\bullet,\bullet}, \partial_{\pi},\bar{\partial})$,
where $\A_{M}^{p,q}$ is the sheaf of $(p,q)$--forms on $M$.
\end{lemma}

\begin{proof}
The lemma is followed by the fact that
$\A_{M}^{p,\bullet}$ is a fine resolution of $\Omega_{M}^{p}$ and $\partial_{\pi}$ commutes with $\bar{\partial}$; also see \cite[Theorem 5.1]{Sti11}.
\end{proof}

\begin{remark}For a complex $n$--dimensional manifold $M$, one has
$$
H_{k}(M, \pi=0)
\cong
\bigoplus\limits_{q-p+n=k}H^{p,q}_{\bar{\partial}}(M).
$$
\end{remark}

Actually the Koszul--Brylinski homology and the Dolbeault cohomology are closely related
with spectral sequences.
\begin{definition}\label{Dol-Poi-spectral-seq}
Let $(M, \pi)$ be a compact holomorphic Poisson manifold of complex dimension $n$.
The following spectral sequence associated to the double complex $(\Gamma(M, \mathcal{A}_{M}^{\bullet,\bullet}), \partial_{\pi},\bar{\partial})$,
\begin{equation}\label{Dol-Poisson-spec.-seq}
E_{1}^{s,t}:=H_{\bar{\partial}}^{n-s,t}(M)
\Longrightarrow H_{n-s+t}(M,\pi),
\end{equation}
is called the {\it Dolbeault--Koszul--Brylinski spectral sequence of $(M, \pi)$}.
\end{definition}

In general, the Dolbeault--Koszul--Brylinski spectral sequence \eqref{Dol-Poisson-spec.-seq}
does not degenerate at $E_{1}$--page (see examples in Section \ref{c6}).
An equivalent condition of the $E_{1}$--degeneracy of Dolbeault--Koszul--Brylinski spectral sequence
is given as follows.

\begin{lemma}[{\cite[Lemma 5.7]{CCYY22}}]\label{deg}
Let $(M, \pi)$ be a compact holomorphic Poisson manifold of complex dimension $n$.
Then the {\it Dolbeault--Koszul--Brylinski spectral sequence} of $(M, \pi)$
degenerates at $E_{1}$--page if and only if  for any $0\leq k\leq 2n$,
$$
\sum_{p-q=n-k} \dim_{\mathbb{C}}\, H_{\bar{\partial}}^{p,q}(M)=\dim_{\mathbb{C}} H_{k}(M,\pi).
$$

\end{lemma}

By a result of Sti\'{e}non \cite[Theorem 6.4]{Sti11},
the holomorphic Evens--Lu--Weinstein pairing on the holomorphic Koszul--Brylinski
homology is non--degenerate.
More precisely,
if $(M, \pi)$ is a compact holomorphic Poisson manifold of complex dimension $n$,
then for $0\leq k\leq 2n$, there is an isomorphism
\begin{equation}\label{Serre-Poincare-duality}
H_{2n-k}(M,\pi)\cong H_{k}(M,\pi).
\end{equation}

Meanwhile, in the dual aspect, there exists a {\it holomorphic Lichnerowicz--Poisson complex}
$(\wedge^{\bullet}\mathcal{T}_{M},b_{\pi})$:
\begin{equation*}
0\to \CO_M\stackrel{b_{\pi}}{\to}
     \cdots \stackrel{b_{\pi}}{\to}\wedge^{s-1}\mathcal{T}_{M} \stackrel{b_{\pi}}{\to}
           \wedge^{s}\mathcal{T}_{M} \stackrel{b_{\pi}}{\to} \wedge^{s+1}\mathcal{T}_{M} \stackrel{b_{\pi}}{\to}
                    \cdots\stackrel{b_{\pi}}{\to}\wedge^{n}\mathcal{T}_{M} \to 0
\end{equation*}
where $b_{\pi}(-)=[\pi,-]_{SN}$.
The $k$--th hypercohomology of $(\wedge^{\bullet}\mathcal{T}_{M},b_{\pi})$ is called the
$k$--th \emph{holomorphic Lichnerowicz--Poisson cohomology}, i.e.,
$$
H^{k}(M, \pi):=\mathbb{H}^{k}(M, (\wedge^{\bullet}\mathcal{T}_{M},b_{\pi})).
$$
If $M$ admits a holomorphic volume form $\omega\in \Gamma(M,\Omega^n_M)$,
that is to say, $M$ is a Calabi--Yau manifold,
then there is a natural morphism of sheaves
$$
\iota_{(-)}\omega:
\wedge^{s}\mathcal{T}_{M}
\to
\Omega_{M}^{n-s}.
$$
However, it does not induce a morphism of sheaf complexes
between
$(\wedge^{\bullet}\mathcal{T}_{M},b_{\pi})$ and $(\Omega_{M}^{\bullet},\partial_{\pi})$
since generally $\iota_{(-)}\omega$ does not commutative with the differentials.
This motivates the following definition.

\begin{definition}[{cf. \cite{We97,BZ99}}]
A holomorphic Poisson manifold $(X,\pi)$ is called {\it unimodular}
if there is a holomorphic volume form $\omega$
such that the morphism  $\iota_{(-)}\omega$
induces a morphism of sheaf complexes from
$(\wedge^{\bullet}\mathcal{T}_{M},b_{\pi})$ to $(\Omega^{\bullet}_{M},\partial_{\pi})$.
\end{definition}
An equivalent condition of
a holomorphic Poisson manifold $(M,\pi)$ being unimodular is $\partial_\pi\omega=0$,
or the modular vector field, introduced by Weinstein \cite{We97} and
Brylinski--Zuckerman \cite{BZ99}, vanishes.
In particular, we have

\begin{proposition}[{\cite[Proposition 4.7]{Sti11}}]\label{dual}
If the $n$--dimensional holomorphic Poisson manifold $(M,\pi)$ is unimodular,
then for any $k\in \mathbb{Z}$, there is an isomorphism
$$
H_{k}(M, \pi)\cong H^{2n-k}(M, \pi).
$$
\end{proposition}

\subsection{$\partial_{}\bar{\partial}$--lemma and $\partial_{\pi}\bar{\partial}$--lemma}

In this subsection,
we consider the $E_{1}$--degeneracy of Dolbeault--Koszul--Brylinski spectral sequence
of a holomorphic Poisson manifold $(M, \pi)$
under the assumption of $\partial_{}\bar{\partial}$--lemma or
$\partial_{\pi}\bar{\partial}$--lemma.
Let $A_{M}^{s,t}:=\Gamma(M, \mathcal{A}_{M}^{s,t})$ be the space of differential $(s,t)$--forms on $M$.
\begin{theorem}\label{deg1}
Let $(M,\pi)$ be a holomorphic Poisson manifold.
If $M$ satisfies the $\partial_{}\bar{\partial}$--lemma,
then its Dolbeault--Koszul--Brylinski spectral sequence degenerates at $E_{1}$--page,
i.e.
$$
H_{k}(M, \pi)
\cong
\bigoplus_{p-q=n-k}H^{p,q}_{\bar{\partial}}(M).
$$
\end{theorem}
\begin{proof}
Recall that a compact complex manifold $M$ satisfies the \emph{$\partial\bar{\partial}$--lemma}, if the equation
$$\ker\,\partial\cap\ker\,\bar{\partial}\cap\mathrm{im}\,d=\mathrm{im}\,\partial\bar{\partial}$$
holds for the double complex $(A^{\bullet,\bullet}_{M}, \partial, \bar{\partial})$ (cf. \cite{DGMS75}).
For any class $[\alpha]\in H^{\bullet,\bullet}_{\bar{\partial}}(M)$,
take $\beta=\partial_{}\alpha$, then we have the following
$$
\left\{ \begin{array}{ll}
\bar{\partial}\beta=\bar{\partial}\partial_{}\alpha=-\partial_{}\bar{\partial}\alpha=0,&\\
\partial_{}\beta=\partial_{}^2\alpha=0, & \\
\beta=d\alpha,  &
\end{array} \right.
$$
i.e. $\beta\in\ker\partial_{}\cap \ker\, \bar{\partial}\cap \mathrm{im} \, d$.
If $M$ satisfies the $\partial_{}\bar{\partial}$--lemma,
then there exists a $\gamma$ on $M$ such that $\beta=\partial_{}\bar{\partial}\gamma$.
Let $\tilde{\alpha}=\alpha-\bar{\partial}\gamma$.
Then we have that $[\alpha]=[\tilde{\alpha}]$ in $H^{\bullet,\bullet}_{\bar{\partial}}(M)$
and $\partial_{}\tilde{\alpha}=0$.
This means that
we can always choose the $\partial$--closed representatives
of the Dolbeault cohomology classes in $H^{\bullet,\bullet}_{\bar{\partial}}(M)$.
Now let $[\alpha]\in H^{\bullet,\bullet}_{\bar{\partial}}(M)$ such that
$\partial(\alpha)=0$. If we write $\eta=\partial_{\pi}(\alpha)$,
then by the Lemma \ref{identities}, we have
$$
\left\{ \begin{array}{ll}
\bar{\partial}\eta=0,&\\
\partial_{}\eta=0, & \\
\eta=(d\iota_{\pi}-\iota_{\pi}d)(\alpha)=d(\iota_{\pi}\alpha),  &
\end{array} \right.
$$
Once again, by the $\partial_{}\bar{\partial}$--lemma of $M$,
there exists a $\zeta$ on $M$ such that
$\partial_{\pi}(\alpha)=\partial_{}\bar{\partial}(\zeta)=-\bar{\partial}\partial_{}(\zeta)$.
Equivalently,
this means the differential of the $E_1$--page of
the Dolbeault--Koszul--Brylinski spectral sequence of $(M, \pi)$
is zero.
Consequently, we have that
the Dolbeault--Koszul--Brylinski spectral sequence of $(M, \pi)$ degenerates at $E_{1}$--page.
Thus
$$
H_{k}(M, \pi)
\cong
H_{k}(M, 0)
\cong
\bigoplus_{p-q=n-k}H^{p,q}_{\bar{\partial}}(M),
$$
and the proof is completed.
\end{proof}

It is worth noting that if $M$ is a projective manifold or K\"{a}hler manifold,
then it automatically satisfies the $\partial\bar{\partial}$--lemma.
Hence a corollary of the Theorem \ref{deg1} is the following.
\begin{corollary}\label{Kahler}
Let $(M,\pi)$ be a compact holomorphic Poisson manifold.
If $M$ is a projective manifold or K\"{a}hler manifold,
then the Dolbeault--Koszul--Brylinski spectral sequence of $(M, \pi)$ degenerates at $E_{1}$--page.
\end{corollary}

Moreover, combined with the Proposition \ref{dual}
and Theorem \ref{deg1}, we have
\begin{corollary}
Let $(M,\pi)$ be a unimodular holomorphic Poisson manifold of complex dimension $n$.
If $M$ satisfies the $\partial\bar{\partial}$--lemma,
then
$$
H^{k}(M, \pi)
\cong
\bigoplus_{p-q=k-n}H^{p,q}_{\bar{\partial}}(M).
$$

\end{corollary}

\begin{example}\label{p-x}
Suppose $\pi_{\mathbb{P}^{n}}$ is a holomorphic Poisson bi--vector field on $\mathbb{P}^{n}$.
By the Corollary \eqref{Kahler},  we have
$$
H_{k}(\mathbb{P}^{n},\pi_{\mathbb{P}^{n}})=
\left\{ \begin{array}{ll}
\mathbb{C}^{n+1},&k=n,\\
0, & k\neq n,
\end{array} \right.$$
since its Hodge numbers are $h^{p,q}=\delta_{pq}$.

Suppose $\pi$ is a holomorphic Poisson bi--vector field on $M=\mathbb{P}^{m}\times \mathbb{P}^{n}$.
By the Corollary \eqref{Kahler}
and the K\"{u}nneth's formula for Dolbeault cohomology(cf. \cite[Corollary 19]{CFGU00}),  we have
$$
H_{k}(M,\pi)=
\left\{ \begin{array}{ll}
\mathbb{C}^{(m+1)(n+1)},&k=m+n,\\
0, & k\neq n.
\end{array} \right.
$$

\end{example}
Motivated by the $\partial\bar{\partial}$--lemma,
following the work \cite{DGMS75}, we have the following definition.
\begin{definition}[{$\partial_{\pi}\bar{\partial}$--lemma}]
Let $(M,\pi)$ be a holomorphic Poisson manifold.
We say that $(M,\pi)$ satisfies the $\partial_{\pi}\bar{\partial}$--lemma if
$$
\ker\, \partial_{\pi}\cap \ker\, \bar{\partial}\cap \mathrm{im} \,(\partial_{\pi}+\bar{\partial})
= \mathrm{im}\,\partial_{\pi} \bar{\partial}.
$$
\end{definition}

An equivalent description of  $\partial_{\pi}\bar{\partial}$--lemma,
which are the special case of \cite[Lemma 5.15]{DGMS75}, states as follows.
\begin{lemma}
Let $(M,\pi)$ be a holomorphic Poisson manifold. Then
the following conditions are equivalent:
\begin{itemize}
\item[(1)]
$(M,\pi)$ satisfies the $\partial_{\pi}\bar{\partial}$--lemma;
\item[(2)]
$\ker\bar{\partial} \cap \mathrm{im}\partial_{\pi}= \mathrm{im}\partial_{\pi}\bar{\partial} ,
\ker\partial_{\pi} \cap \mathrm{im}\bar{\partial}= \mathrm{im}\partial_{\pi}\bar{\partial}$;
\item[(3)]
$\ker\bar{\partial} \cap \ker\partial_{\pi} \cap (\mathrm{im}\partial_{\pi}+\mathrm{im}\bar{\partial})= \mathrm{im}\partial_{\pi}\bar{\partial}.$
\end{itemize}
\end{lemma}

Another two closely related cohomologies we need in this paper are
the Bott--Chern cohomology and the Aeppli cohomology.
These two cohomologies are the special case of Angella and Tomassini \cite{AT15}.
\begin{definition}
Let $(M,\pi)$ be a holomorphic Poisson manifold.
The {\it $(p,q)$--th Bott--Chern cohomology} of $(M,\pi)$ is defined as
$$
H_{BC}^{p,q}(M,\pi):=
\frac{\ker\,\partial_{\pi}\cap \ker\,\bar{\partial}}{\mathrm{im}\, \partial_{\pi}\bar{\partial}},
$$
while the {\it $(p,q)$--th Aeppli cohomology} of $(M,\pi)$ is defined as
$$
H_{A}^{p,q}(M,\pi):=
\frac{\ker\,\partial_{\pi}\bar{\partial}}{\mathrm{im}\, \partial_{\pi}+\mathrm{im}\, \bar{\partial}}.
$$
\end{definition}
One can check that the identity map induces natural morphisms
\begin{equation}\label{morphisms}
\vcenter{
\xymatrix@R=0.5cm{
                &         H_{BC}^{\bullet,\bullet}(M,\pi) \ar[dl]^{} \ar[d]^{}\ar[dr]^{}  & \\
  H_{\partial_{\pi}}^{\bullet,\bullet}(M) \ar[dr]_{} & H_{\bullet}(M, \pi)\ar[d]^{} &   H_{\bar{\partial}}^{\bullet,\bullet}(M) \ar[dl]^{}    \\
                &         H_{A}^{\bullet,\bullet}(M,\pi)   &            }
}
\end{equation}
since $\ker\, \partial_{\pi}\cap \ker\, \bar{\partial}\subset\ker\,(\partial_{\pi}+\bar{\partial})
\subset\ker\,\partial_{\pi}\bar{\partial}$
and
$\mathrm{im}\,\partial_{\pi} \bar{\partial}
\subset \mathrm{im}\,(\partial_{\pi}+ \bar{\partial})
\subset \mathrm{im}\,\partial_{\pi}+ \mathrm{im}\,\bar{\partial}$.
Generally, each morphism in \ref{morphisms} is neither injective nor surjective.
The following theorem is an application of \cite[Theorem 1 \& Theorem 2 \& Lemma 2.4]{AT15}.

\begin{theorem}\label{lem1}
Let $(M,\pi)$ be a holomorphic Poisson manifold. Then
\begin{itemize}
\item[(1)]
$\sum\limits_{p+q=k}\big(\mathrm{dim}_{\mathbb{C}} H_{BC}^{p,q}(M,\pi)
+\mathrm{dim}_{\mathbb{C}}H_{A}^{p,q}(M,\pi)\big)
\geq
2\mathrm{dim}_{\mathbb{C}}H_{k}(M,\pi)$;
\item[(2)]the identity holds if and only if $(M,\pi)$ satisfies the $\partial_{\pi}\bar{\partial}$--lemma.
In this case, all morphisms in the diagram \ref{morphisms} are isomorphic.
\end{itemize}
\end{theorem}

For more characterizations of the $\partial_{\pi}\bar{\partial}$--lemma
we refer to \cite[Proposition 5.17]{DGMS75}.
Especially,
an analogous result to Theorem \ref{deg1}
whose proof is obtained by replacing the operator $\partial_{}$ with $\partial_{\pi}$
states as follows.

\begin{theorem}\label{deg2}
Let $(M,\pi)$ be a compact holomorphic Poisson manifold.
If $(M,\pi)$ satisfies the $\partial_{\pi}\bar{\partial}$--lemma,
then the Dolbeault--Koszul--Brylinski spectral sequence $E_{\bullet}$
degenerate at the first page, or
$$
H_{k}(M, \pi)
\cong
\bigoplus_{p-q=n-k}H^{p,q}_{\partial_{\pi}}(M).
$$
\end{theorem}

\begin{remark}
One can check that
for a complex manifold $M$ with trivial holomorphic Poisson bi--vector field $\pi$,
it does not satisfy the $\partial_{\pi}\bar{\partial}$--lemma,
but its Dolbeault--Koszul--Brylinski spectral sequence degenerates at the first page.
In Section \ref{c6} more examples
whose Dolbeault--Koszul--Brylinski spectral sequence degenerate at $E_{1}$--page
but do not satisfy the $\partial_{\pi}\bar{\partial}$--lemma  are given.
\end{remark}

\section{Formality of the Dolbeault complex}\label{c3}
Recall for a holomorphic Poisson manifold $(M, \pi)$,
the Koszul--Brylinski operator $\partial_{\pi}$ is a BV operator,
and it generates a Lie bracket $[\alpha, \beta]_{\partial_{\pi}}$.
In this section we consider its formality properties with respect to variant differentials.

\begin{lemma}\label{DGL}
Suppose $(M, \pi)$ is a holomorphic Poisson manifold.
Then $(A_{M}^{\bullet,\bullet},\partial,[-,-]_{\partial_\pi})$,
$(A_{M}^{\bullet,\bullet},\bar{\partial},[-,-]_{\partial_\pi})$
and
$(A_{M}^{\bullet,\bullet},d,[-,-]_{\partial_\pi})$
are three differential Gerstenhaber algebras.
\end{lemma}

\begin{proof}
Note
$(A_{M}^{\bullet,\bullet},[-,-]_{\partial_\pi})$
is a Lie algebra.
By the Lemma \ref{identities},
all three differential $\partial, \bar{\partial}$ and $d$ are derivations
with respect to the bracket $[-,-]_{\partial_\pi}$,
thus we have the lemma.
\end{proof}

For any $k\geq1$, the map
\begin{equation}\label{contraction}
  \iota_{\pi}^k: (A_{M}^{\bullet,\bullet},\partial_{\pi},\bar{\partial})
  \to (A_{M}^{\bullet-2k,\bullet},\partial_{\pi},\bar{\partial})
\end{equation}
is a well--defined morphism of double complexes
since by the Lemma \ref{identities},
$\iota_{\pi}^{k}\partial_{\pi}=\partial_{\pi}\iota_{\pi}^{k}$
and $\iota_{\pi}^{k}\bar{\partial}=\bar{\partial}\iota_{\pi}^{k}$.
Thus the operator
$$e^{\iota_{\pi}}:=\sum_{k=0}\frac{1}{k!}\iota_{\pi}^k,$$
is well--defined on $(A_{M}^{\bullet,\bullet},\partial_{\pi},\bar{\partial})$ with inverse $e^{-\iota_{\pi}}$.
\begin{lemma}\label{f1}
For any natural number $k$,
$\iota_{\pi}^k\partial=\partial\iota_{\pi}^k+k\iota_{\pi}^{k-1}\partial_{\pi}$.
Moreover, we have that
\begin{equation}\label{f2}
  e^{\iota_{\pi}}\partial =(\partial+\partial_\pi)e^{\iota_{\pi}}.
\end{equation}
\end{lemma}
\begin{proof}
Inductively,
$$
\iota_{\pi}^{k+1}\partial
=\iota_{\pi}(\partial\iota_{\pi}^k+k\iota_{\pi}^{k-1}\partial_{\pi})
=\iota_{\pi}\partial\iota_{\pi}^k+k\iota_{\pi}^{k}\partial_{\pi}
=\partial\iota_{\pi}^{k+1}+(k+1)\iota_{\pi}^{k}\partial_{\pi}.
$$
Thus, we have
\begin{center}
$
e^{\iota_{\pi}}\partial
=\sum\limits_{k=0}\frac{1}{k!}\iota_{\pi}^k\partial
=\sum\limits_{k=0}\frac{1}{k!}(\partial\iota_{\pi}^k+k\iota_{\pi}^{k-1}\partial_{\pi})
=(\partial+\partial_\pi)e^{\iota_{\pi}},
$
\end{center}
and the lemma is proved.
\end{proof}
\begin{theorem}[{\cite[Corollary 2]{ST08}} \& {\cite[Theorem 3.2]{FM12}}]
The DGLAs
$(A_M^{\bullet,\bullet},\partial,[-,-]_{\partial_\pi})$
and
$(A_M^{\bullet,\bullet},d,[-,-]_{\partial_\pi})$
are formal
and quasi--isomorphic to abelian Lie algebras.
\end{theorem}
\begin{proof}
For any $\partial$--closed forms $\alpha\in A_M^{i,k-i},\ \beta\in A_M^{j,l-j}$,
\begin{eqnarray*}
[\alpha,\beta]_{\partial_{\pi}} &=&
 (-1)^k\big(\partial_{\pi}(\alpha\wedge \beta)-(\partial_{\pi}\alpha) \wedge \beta
               -(-1)^{k} \alpha\wedge (\partial_{\pi}\beta)\big)\\
     &=& \partial\big((-1)^k\iota_{\pi}(\alpha\wedge \beta)-(-1)^k(\iota_{\pi}\alpha) \wedge \beta
               - \alpha\wedge (\iota_{\pi}\beta)\big),
\end{eqnarray*}
i.e. $[\alpha,\beta]_{\partial_{\pi}}$ is $\partial$--exact.
Combined with the Lemma \eqref{f1}, we have that
$(A_M^{\bullet,\bullet},\partial,[-,-]_{\partial_\pi})$
is formal and quasi--isomorphic to the abelian Lie algebra $(H^{\bullet,\bullet}_\partial(M),0)$.

Set $d_{\pi}:=\iota_{\pi}d-d\iota_{\pi}=\partial_{\pi}+\bar{\partial}_{\pi}$.
According to a result by Sharygin--Talalaev \cite[Lemma 5]{ST08},
the Lie bracket
$$
[\alpha,\beta]_{d_{\pi}}=
(-1)^k\big(d_{\pi}(\alpha\wedge \beta)-(d_{\pi}\alpha) \wedge \beta
               -(-1)^{k} \alpha\wedge (d_{\pi}\beta)\big),
               \alpha\in A_{M}^{k},\ \beta\in A_{M}^{l},
$$
associated to $d_{\pi}$ is $d$--exact
if both $\alpha$ and $\beta$ are $d$--closed.
Since the Poisson bi--vector field $\pi$ is  holomorphic,
one have that
the Dolbeault operator $\bar{\partial}$ commutes with the operator $\iota_{\pi}$.
Equivalently, we get $\bar{\partial}_{\pi}=\iota_{\pi}\bar{\partial}-\bar{\partial}\iota_{\pi}=0$.
Therefore we in fact have that $[\alpha,\beta]_{\partial_{\pi}}=[\alpha,\beta]_{d_{\pi}}$ is $d$--exact,
and the equation \eqref{f2} becomes as
\begin{eqnarray*}
 e^{\iota_{\pi}}d&=& e^{\iota_{\pi}}(\partial+\bar{\partial}) \\
 &=& e^{\iota_{\pi}}\partial+e^{\iota_{\pi}}\bar{\partial} \\
   &=& (\partial+\partial_\pi)e^{\iota_{\pi}} +\bar{\partial}e^{\iota_{\pi}}\\
   &=& (d+\partial_\pi)e^{\iota_{\pi}}.
\end{eqnarray*}
Therefore we conclude that
$(A_M^{\bullet,\bullet},d,[-,-]_{\partial_\pi})$
is formal
and quasi--isomorphic to the abelian Lie algebra $(H^{\bullet}_{dR}(M),0)$.
\end{proof}

We are now in a position to give the proof of Theorem \ref{main-theorem2}.
\begin{proof}[Proof of Theorem \ref{main-theorem2}]
Under the assumption that $(M,\pi)$ satisfies the $\partial_{\pi}\bar{\partial}$--lemma,
we claim the morphisms of DGLAs in the diagram
  \begin{equation*}\label{formality}
 \begin{tikzpicture}[scale=0.3]
\draw (-10,0) node[left] {$(A_{M}^{\bullet,\bullet},\bar{\partial},[-,-]_{\partial_\pi})$};
   \draw[-latex] (-8,0)-- (-10,0) node[right] {\quad\, $(\ker\partial_{\pi},\bar{\partial},[-,-]_{\partial_\pi})$};
   \draw(6.5,0) node[right] {$(\ker\bar{\partial},\partial_{\pi},[-,-]_{\partial_\pi})$};
   \draw[-latex] (18,0)-- (20,0) node[right] {$(A_{M}^{\bullet,\bullet},\partial_{\pi},[-,-]_{\partial_\pi}) $};
\draw  (3.5,-5) node[left] {$(H_{\partial_{\pi}}^{\bullet,\bullet}(M),0,[-,-]_{\partial_\pi}) $};
\draw(6.5,-5) node[right] {$(H_{\bar{\partial}}^{\bullet,\bullet}(M),0,[-,-]_{\partial_\pi}) $};
\draw  (12.5,-10) node[left] {$(H_{BC}^{\bullet,\bullet}(M,\pi),0,[-,-]_{\partial_\pi})$};
 \draw[-latex] (-2.8,-0.5)-- (-2.8,-4);
 \draw[-latex] (3,-8.5)--(-3,-6);
  \draw[-latex] (13,-1)-- (13,-4);
 \draw[-latex] (7,-8.5)--(13,-6);
      \draw  (-2.5,-2) node[left] {$p_1$};
            \draw  (13,-2) node[left] {$p_2$};
                  \draw  (-8,1) node[left] {$i_1$};
            \draw  (20,1) node[left] {$i_2$};
          \end{tikzpicture}
\end{equation*}
are all well--defined and quasi--isomorphic.
In fact, by the symmetry of the two operators $\partial_{\pi}$ and $\bar{\partial}$
in $\partial_{\pi}\bar{\partial}$--lemma
and the Theorem \ref{lem1},
we only need to check the morphisms $i_1$ and $p_1$ are well--defined and quasi--isomorphic.

Indeed,
by the facts that $\partial_{\pi}$ is commutative with $\bar{\partial}$
and both $\partial_{\pi}, \bar{\partial}$ are derivations of the bracket $[-,-]_{\partial_\pi}$,
we obtain that both $i_1$ and $p_1$ are well--defined.
Following \cite[Proposition 9.7.1]{Man99},
we prove that both $i_1$ and $p_1$ are quasi--isomorphisms.

If $\partial_{\pi}(\alpha)=0, i_1(\alpha)=\alpha=\bar{\partial} (\beta)$ for some $\beta$,
then $\alpha\in \ker\partial_{\pi}\cap \mathrm{im}\bar{\partial}$.
By the $\partial_{\pi}\bar{\partial}$--lemma,
there exists a $\gamma$ such that $\alpha=\bar{\partial}\partial_{\pi}(\gamma)$.
This means that for any element $\alpha$ of $\ker\partial_{\pi}$,
if its image $i_1(\alpha)$ is $\bar{\partial}$--exact,
then $\alpha$ itself is $\bar{\partial}$--exact in $\ker\partial_{\pi}$.
By definition, we conclude that $H(i_1)$ is injective.

Further,
if $\bar{\partial} (\alpha)=0$,
then $\partial_{\pi}(\alpha)\in \ker\bar{\partial} \cap \mathrm{im}\partial_{\pi}$.
By the $\partial_{\pi}\bar{\partial}$--lemma,
there exists a $\beta$ such that $\partial_{\pi}(\alpha)=\partial_{\pi}\bar{\partial}(\beta)$.
Equivalently,
$\alpha-\bar{\partial}\beta\in \ker \partial_{\pi}$
and $[\alpha-\bar{\partial}\beta]=[\alpha]$ in $H_{\bar{\partial}}^{\bullet,\bullet}(M)$.
That is to say, $H(i_1)$ is surjective.

Meanwhile, if $\alpha\in \mathrm{im}\partial_{\pi}$ and $\bar{\partial}(\alpha)=0$,
then $\alpha\in \ker\bar{\partial}\cap \mathrm{im}\partial_{\pi}$.
Once again, by the $\partial_{\pi}\bar{\partial}$--lemma,
there exists a $\gamma$ such that $\alpha=\bar{\partial}\partial_{\pi}(\gamma)$.
In other words, $H(p_1)$ is injective.

At last, the morphism $H(p_1)$ is surjective
since if $\partial_{\pi}(\alpha)=0$,
then $\bar{\partial}\alpha\in \ker\partial_{\pi}\cap \mathrm{im}\bar{\partial}$.
Once again, by the $\partial_{\pi}\bar{\partial}$--lemma,
there exists a $\beta$ such that $\bar{\partial}\alpha=\bar{\partial}\partial_{\pi}(\beta)$.
Equivalently, the $\bar{\partial}$--closed element $\alpha-\partial_{\pi}(\beta)$
is cohomologous to $\alpha$ in $H_{\partial_{\pi}}^{\bullet,\bullet}(M)$.

Thus,
if $(M,\pi)$ satisfies the $\partial_{\pi}\bar{\partial}$--lemma,
then the DGLA
$(A_M^{\bullet,\bullet},\bar{\partial},[-,-]_{\partial_\pi})$
is quasi--isomorphic to $(H_{\bar{\partial}}^{\bullet,\bullet}(M),0,[-,-]_{\partial_\pi})$.
By definition, this means that $(A_M^{\bullet,\bullet},\bar{\partial},[-,-]_{\partial_\pi})$ is formal.
\end{proof}

\section{Maurer--Cartan elements}\label{c4}

In this section we
consider the Maurer--Cartan equation of the DGLA
$(A_M^{\bullet,\bullet}[[t]]=A_M^{\bullet,\bullet}\otimes \mathbb{C}[[t]],\bar{\partial},[-,-]_{\partial_\pi})$:
\begin{equation}\label{MCE1}
  \bar{\partial}\alpha_t+\frac{1}{2}[\alpha_t,\alpha_t]_{\partial_\pi}=0.
\end{equation}
Naturally, due to the degree reason, the solutions (called Maurer--Cartan elements)
of such equation (if exists) lie in $A_M^{1,1}[[t]]$.
If we write $\alpha_t=\sum\limits_{i=1}^\infty\alpha_it^i$,
then the Maurer--Cartan equation is equivalent to
a system of equations
\begin{equation}\label{MCE2}
\begin{cases}
  \bar{\partial}\alpha_1=0 ,   &   \\
  \bar{\partial}\alpha_k+\frac{1}{2}\sum\limits_{i=1}^{k-1}[\alpha_i,\alpha_{k-i}]_{\partial_\pi}=0,   & k\geq 2.
\end{cases}
\end{equation}

\begin{theorem}\label{Solution_of_MCE1}
Let $(M,\pi)$ be a holomorphic Poisson manifold.
If $(M,\pi)$ satisfies the $\partial_{}\bar{\partial}$--lemma
or $\partial_{\pi}\bar{\partial}$--lemma,
then for any $[\alpha]\in H_{\bar{\partial}}^{1,1}(M)$,
there exists a Maurer--Cartan element
$\alpha_t$
whose $\alpha_1$ is a representative of $[\alpha]$.
\end{theorem}
\begin{proof}
We here only prove the theorem
under the assumption of $\partial_{\pi}\bar{\partial}$--lemma
since the case that $(M,\pi)$ satisfies the $\partial_{}\bar{\partial}$--lemma
can be obtained with total same strategy by replacing $\partial_{\pi}$ with $\partial$.

Assume that $(M,\pi)$ satisfies the $\partial_{\pi}\bar{\partial}$--lemma.
To prove the theorem,
it is sufficient to find $\alpha_2,\alpha_3,\cdots$ such that they satisfy the
Maurer--Cartan equation \ref{MCE2}.
For any class $[\alpha]\in H^{\bullet,\bullet}_{\bar{\partial}}(M)$,
take $\beta=\partial_{\pi}\alpha$, then we have the following
$$
\left\{ \begin{array}{ll}
\bar{\partial}\beta=\bar{\partial}\partial_{\pi}\alpha=-\partial_{\pi}\bar{\partial}\alpha=0,&\\
\beta=\partial_{\pi}\alpha,  &
\end{array} \right.
$$
i.e. $\beta\in\ker\partial_{\pi}\cap \ker\, \bar{\partial}
\cap (\mathrm{im} \,\partial_{\pi}+\mathrm{im} \,\bar{\partial})$.
By the $\partial_{\pi}\bar{\partial}$--lemma of $(M,\pi)$,
there exists a $\gamma$ on $M$ such that $\beta=\partial_{\pi}\bar{\partial}\gamma$.
Let $\tilde{\alpha}=\alpha-\bar{\partial}\gamma$.
Then we have that $[\alpha]=[\tilde{\alpha}]$ in $H^{\bullet,\bullet}_{\bar{\partial}}(M)$
and $\partial_{\pi}\tilde{\alpha}=0$.
Therefore, in what follows
we always choose the $\partial_{\pi}$--closed representatives
of the Dolbeault cohomology classes in $H^{1,1}_{\bar{\partial}}(M)$.
Let $[\alpha_1]\in H_{\bar{\partial}}^{1,1}(M)$ such that
$\partial_{\pi}(\alpha_1)=0$.,
then
$$
\left\{ \begin{array}{ll}
\bar{\partial}[\alpha_1,\alpha_1]_{\partial_\pi}
=[\bar{\partial}\alpha_1,\alpha_1]_{\partial_\pi}
-[\alpha_1,\bar{\partial}\alpha_1]_{\partial_\pi}=0,&\\
{[\alpha_1,\alpha_1]_{\partial_\pi}}
=\partial_\pi(\alpha_1\wedge\alpha_1)
-\partial_\pi(\alpha_1)\wedge\alpha_1
-\alpha_1\wedge\partial_\pi(\alpha_1)=\partial_\pi(\alpha_1\wedge\alpha_1),&
\end{array} \right.
$$
This means that $[\alpha_1,\alpha_1]_{\partial_\pi}$ is $\bar{\partial}$--closed and $\partial_\pi$--exact.
By the $\partial_{\pi}\bar{\partial}$--lemma of $(M,\pi)$,
there exists a $(2,1)$--form $\zeta_2$ such that $[\alpha_1,\alpha_1]_{\partial_\pi}=\partial_\pi\bar{\partial}\zeta_2$.
Therefore if take $\beta_2=\frac{1}{2}\zeta_2,\alpha_2=\partial_\pi\beta_2$,
then we have that
$$\bar{\partial}\alpha_2+\frac{1}{2}[\alpha_1,\alpha_1]_{\partial_\pi}=0.$$
Inductively, suppose that
we already found $\partial_\pi$--exact forms
$\alpha_2=\partial_\pi\beta_2,\cdots,\alpha_k=\partial_\pi\beta_k$ satisfying the Maurer--Cartan equation \ref{MCE2}.
Let
$$
\gamma_k=[\alpha_1,\alpha_{k}]_{\partial_\pi}+\cdots+[\alpha_i,\alpha_{k+1-i}]_{\partial_\pi}
+\cdots+[\alpha_k,\alpha_{1}]_{\partial_\pi}.
$$
Note the degree of the bracket
$[-,-]_{\partial_\pi}$ is $-1$,
and the symmetry of $[-,-]_{\partial_\pi}$ indicates that
$[\alpha_i,\alpha_{j}]_{\partial_\pi}=[\alpha_j,\alpha_{i}]_{\partial_\pi}$
and $[\bar{\partial}\alpha_i,\alpha_{j}]_{\partial_\pi}=-[\alpha_j,\bar{\partial}\alpha_{i}]_{\partial_\pi}$.
When $k=2l\geq4$, we have that
\begin{eqnarray*}
  \bar{\partial}\gamma_k
  &=& 2\cdot\sum\limits_{i=1}^{l}\bar{\partial} [\alpha_i,\alpha_{2l+1-i}]_{\partial_\pi} \\
   &=&  2\cdot\sum\limits_{i=1}^{l}\big( [\bar{\partial}\alpha_i,\alpha_{2l+1-i}]_{\partial_\pi}
                                  +[\bar{\partial}\alpha_{2l+1-i},\alpha_i]_{\partial_\pi}\big)\\
   &=&\sum\limits_{s+t=i}\sum\limits_{i=2}^{l}
                   [[\alpha_s,\alpha_{t}]_{\partial_\pi},\alpha_{2l+1-i}]_{\partial_\pi}
       +\sum\limits_{p+q=2l+1-i}\sum\limits_{i=1}^{l}
                          [[\alpha_p,\alpha_{q}]_{\partial_\pi},\alpha_i]_{\partial_\pi}\\
   &=& \sum\limits_{p+q+r=2l+1}[[\alpha_p,\alpha_{q}]_{\partial_\pi},\alpha_r]_{\partial_\pi}\\                             &=& 0.
\end{eqnarray*}
Analogous, when $k=2l-1\geq3$, we have that
\begin{eqnarray*}
  \bar{\partial}\gamma_k &=& 2\cdot\sum\limits_{i=1}^{l-1}\bar{\partial} [\alpha_i,\alpha_{2l-i}]_{\partial_\pi}
  +\bar{\partial} [\alpha_l,\alpha_{l}]_{\partial_\pi}\\
   &=&  2\cdot\sum\limits_{i=1}^{l-1}\big( [\bar{\partial}\alpha_i,\alpha_{2l-i}]_{\partial_\pi}
            +[\bar{\partial}\alpha_{2l-i},\alpha_i]_{\partial_\pi} \big)
               + 2\cdot[\bar{\partial}\alpha_l,\alpha_{l}]_{\partial_\pi}\\
 &=&\sum\limits_{s+t=i}\sum\limits_{i=2}^{l-1}
             [[\alpha_s,\alpha_{t}]_{\partial_\pi},\alpha_{2l-i}]_{\partial_\pi}
+\sum\limits_{p+q=2l-i}\sum\limits_{i=1}^{l-1}[[\alpha_p,\alpha_{q}]_{\partial_\pi},\alpha_i]_{\partial_\pi}
+\sum\limits_{u+v=l}[[\alpha_u,\alpha_{v}]_{\partial_\pi},\alpha_l]_{\partial_\pi}\\
&=&
\sum\limits_{p+q+r=2l}[[\alpha_p,\alpha_{q}]_{\partial_\pi},\alpha_r]_{\partial_\pi}\\                             &=& 0.
\end{eqnarray*}
Hence we always have that $\gamma_k$ is $\bar{\partial}$--closed.
Moreover,
$$
\gamma_k
=[\alpha_1,\partial_\pi\beta_{k}]_{\partial_\pi}+
\sum\limits_{i=2}^{k} [\partial_\pi\beta_i,\alpha_{k+1-i}]_{\partial_\pi}
=\partial_\pi\big(-[\alpha_1,\beta_{k}]_{\partial_\pi}+\sum\limits_{i=2}^{k} [\beta_i,\alpha_{k+1-i}]_{\partial_\pi}\big),
$$
This means that $\gamma$ is also $\partial_\pi$--exact.
By the $\partial_{\pi}\bar{\partial}$--lemma of $(M,\pi)$,
there exists a $(2,1)$--form $\zeta_{k+1}$ such that $\gamma=\partial_\pi\bar{\partial}\zeta_{k+1}$.
Therefore if take $\beta_{k+1}=\frac{1}{2}\zeta_{k+1},\alpha_{k+1}=\partial_\pi\beta_{k+1}$,
then we have that
$$\bar{\partial}\alpha_{k+1}+\frac{1}{2}\sum\limits_{i=1}^{k}[\alpha_1,\alpha_1]_{\partial_\pi}=0.$$
So the theorem is proved.
\end{proof}

Recall for a holomorphic Poisson manifold $(M, \pi)$,
the holomorphic Poisson bi--vector field $\pi$ induces
a sheaf morphism
$\pi^\sharp:\Omega_{M}^{1}\to \mathcal{T}_{M}$ by contraction
with $\pi$.
More generally, for any $p\geq1$, $\pi$ induces
a sheaf morphism
$
\pi^\sharp:\Omega_{M}^{p}\to\wedge^p\mathcal{T}_{M}
$
and
$$
\pi^\sharp:\A_{M}^{p,q}\to\A^{0,q}(M,\wedge^p\mathcal{T}_{M})
$$
which is given locally by
$$
\pi^\sharp:dz_{i_1}\wedge\cdots\wedge dz_{i_p}\wedge d\hat{z}_{i_1}\wedge\cdots\wedge d\hat{z}_{i_q}
\mapsto (-1)^pX_{z_{i_1}}\wedge\cdots\wedge X_{z_{i_p}}\otimes d\hat{z}_{i_1}\wedge\cdots\wedge d\hat{z}_{i_q}.
$$
Here $X_f$ is the Hamiltonian vector field with respect to $f$.
Such $\pi^\sharp$ connects the DGLA $(A_{M}^{\bullet,\bullet},\bar{\partial},[-,-]_{\partial_\pi})$
with the Kodaira--Spencer DGLA $(A^{0,\bullet}(M,\wedge^\bullet T_{M}),\bar{\partial},[-,-]_{})$.
More precisely, we have
\begin{proposition}\label{deformation}
Let $(M, \pi)$ be a holomorphic Poisson manifold,
then
the holomorphic Poisson bi--vector field $\pi$ induces
a map of DGLAs
$$
\pi^\sharp:(A_{M}^{\bullet,\bullet},\bar{\partial},[-,-]_{\partial_\pi})
\to (A^{0,\bullet}(M,\wedge^\bullet T_{M}),\bar{\partial},[-,-]_{}).
$$
\end{proposition}
\begin{proof}
By definition, we have that
$\pi^\sharp$ is homomorphism of algebras
$A_{M}^{\bullet,\bullet}$
and $A^{0,\bullet}(M,\wedge^\bullet T_{M})$.
Since $\pi$ is holomorphic,
one can obtain that $\pi^\sharp$ is commutative with $\bar{\partial}$.
At last, note these two DGLAs are Gerstenhaber algebras,
and for any $1$--forms $\alpha,\beta$,
$$\pi^\sharp[\alpha,\beta]_{\partial_\pi}
=[\pi^\sharp(\alpha),\pi^\sharp(\beta)]_{},
$$
thus the map $\pi^\sharp$ is a homomorphism of Lie algebras
and the proposition is proved.
\end{proof}

A classical result of Maurer--Cartan equation
is that a map of DGLAs induces a map of Maurer--Cartan elements.
By this fact, based on the Proposition \ref{deformation},
a corollary of the Theorem \ref{Solution_of_MCE1} states:
\begin{corollary}\label{def}
If a holomorphic Poisson manifold $(X,\pi)$ satisfies the $\partial_{}\bar{\partial}$--lemma
or $\partial_{\pi}\bar{\partial}$--lemma,
then for any $[\alpha]\in H_{\bar{\partial}}^{1,1}(M)$,
$\pi^\sharp[\alpha]\in H^1(M,\mathcal{T}_{M})$
is tangent to a deformation of complex structure.
\end{corollary}


\section{Examples}\label{c6}


In this section, we discuss some examples
with different properties in the viewpoints of Poisson geometry.
More precisely we consider some special nilmanifolds and solvmanifolds.

Let $G$ be a complex nilpotent Lie group with Lie algebra $\mathfrak{g}$
whose complexification is $\mathfrak{g}_{\mathbb{C}}:=\mathfrak{g}\mathfrak{}\otimes_{\mathbb{R}}\mathbb{C}$,
and let $H$ be a co--compact discrete subgroup of $G$.
Suppose $M=G/H$ is the associated nilmanifold endowed with a $G$--left--invariant
complex structure $J$ and an $G$--left--invariant holomorphic Poisson bi--vector field $\pi$.
Then for any $p$, there exists a natural inclusion of complexes
$$
i: (\wedge^{p,\bullet}\mathfrak{g}_{\mathbb{C}}^*,\bar{\partial})
\hookrightarrow
(\Gamma(M,\A_M^{p,\bullet}),\bar{\partial}).
$$

\begin{lemma}[{\cite[Lemma 6.1]{CCYY22}}]\label{nil}
If the map $i$ defined above is a quasi--isomorphism, then the $k$--th total cohomology of double complex $(\wedge^{\bullet,\bullet}\mathfrak{g}_{\mathbb{C}}^*,\partial_\pi,\bar{\partial})$
is isomorphic to $H_k(M,\pi)$ for any $k$.
\end{lemma}

\begin{remark}
A result of Sakane (see \cite[Theorem 1]{Sak76}) states that
if a nilmanifold is complex parallelizable (i.e. the holomorphic tangent bundle is
holomorphically trivial),
the inclusion $i$ is a quasi--isomorphism.
For more conditions such that the inclusion $i$ is a quasi--isomorphism,
the reader may refer to Angella \cite[Theorem 3.2]{Angella}.
\end{remark}
\subsection{Iwasawa manifold}

Let $\mathrm{H}(3; \mathbb{C})$ be the complex Heisenberg Lie group:
$$
\mathrm{H}(3; \mathbb{C})=
\Bigg\{ \small{\left(\begin{array}{ccc}
1 & z_{1} & z_{2}\\
0 & 1  & z_{3} \\
0 & 0 & 1
\end{array}
\right)}\mid z_{i}\in \mathbb{C} \Bigg\}\subset \mathrm{GL}(3; \mathbb{C}).
$$
As complex manifolds, $\mathrm{H}(3; \mathbb{C})$ is isomorphic to $\mathbb{C}^{3}$.
Consider the discrete group
$\mathrm{G}_3:=\mathrm{GL}(3; \mathbb{Z}[i])\cap \mathrm{H}(3; \mathbb{C})$,
where $\mathbb{Z}[i]=\{a+bi\mid a,b\in \mathbb{Z}\}$ is the Gaussian integers.
The left multiplication gives a natural $\mathrm{G_3}$--action on $\mathrm{H}(3; \mathbb{C})$,
and a correspondent faithful $\mathrm{G}_3$--action on $\mathbb{C}^{3}$ given by
$
(a_{1}, a_{2}, a_{3})\cdot(z_{1}, z_{2}, z_{3}):=(z_{1}+a_{1}, z_{2}+a_{1}z_{3}+a_{2}, z_{3}+a_{3}),
$
where $a_{1}, a_{2}, a_{3}\in \mathbb{Z}[i]$.
Consequently, the associated $\mathrm{G}_3$--quotient space
$$
\mathbb{I}_3:=\mathbb{C}^{3}/ \mathrm{G}_3
$$
is a compact complex threefold, which is called {\it Iwasawa manifold}.

A basis of the space $(\mathfrak{g}_{\mathbb{C}}^*)^{1,0}$ of
left--invariant holomorphic differential forms of $\mathrm{H}(3; \mathbb{C})$ is
$$w_1=dz_1, w_2=dz_2-z_1dz_3, w_3=dz_3.$$
Therefore $dw_1=dw_3=0, dw_2=-w_1\wedge w_3,
\bar{\partial}(\bar{w_1})=\bar{\partial}(\bar{w_3})=0,
\bar{\partial}(\bar{w_2})=-\bar{w_1}\wedge \bar{w_3}.$
Dually, a basis of Lie algebra $\mathfrak{g}_{\mathbb{C}}^{1,0}$ of
left--invariant holomorphic vector fields of $\mathrm{H}(3; \mathbb{C})$ is
$$X_1=\frac{\partial}{\partial z_1},
X_2=\frac{\partial}{\partial z_2},
 X_3=\frac{\partial}{\partial z_3}+z_1\frac{\partial}{\partial z_2}$$
with relations $[X_1,X_2]=[X_2,X_3]=0, [X_1,X_3]=X_2$.
As a consequent, $\mathbb{I}_3$ is complex parallelizable,
hence it is a Calabi--Yau manifold.

Furthermore, a $\mathrm{H}(3; \mathbb{C})$--left--invariant holomorphic bi--vector field is in the form of
$\pi=c_1X_1\wedge X_2+c_2X_1\wedge X_3+c_3X_2\wedge X_3$.
One can check that $[\pi,\pi]=0$ if and only if $c_2=0$.
Note $\pi$ is the linear combination of two compatible Poisson bi--vector fields
  $\pi_{12}=X_1\wedge X_2$ and $\pi_{23}=X_2\wedge X_3$.
But on $\wedge^{\bullet,\bullet}\mathfrak{g}^*_{\mathbb{C}}$,
$\partial_{\pi_{12}}=\partial_{\pi_{23}}=0$ which indicates $\partial_\pi=0.$
Hence by Lemma \ref{nil},
the Dolbeault--Koszul--Brylinski spectral sequence of $(\mathbb{I}_3,\pi)$ degenerates at $E_1$--page.
However $(\mathbb{I}_3,\pi)$ neither satisfies the $\partial_{}\bar{\partial}$--lemma
nor satisfies the $\partial_{\pi}\bar{\partial}$--lemma.
Indeed, by the structure equations,
we have that
$$-w_1\wedge\bar{w_1}\wedge \bar{w_3}\notin\mathrm{im}\partial_{}\bar{\partial},$$
 while
$$
-w_1\wedge\bar{w_1}\wedge \bar{w_3}
=(\partial_{}+\bar{\partial})(w_1\wedge\bar{w_2})
\in\ker\, \partial_{}\cap \ker\, \bar{\partial}\cap \mathrm{im} \,(\partial_{}+\bar{\partial})
$$
which indicates
that $(\mathbb{I}_3,\pi)$ does not satisfy the $\partial_{}\bar{\partial}$--lemma
(see also \cite{Angella13}).
Meanwhile,
since $\mathrm{im}\partial_{\pi}\bar{\partial}=\emptyset$,
$$
-\bar{w_1}\wedge \bar{w_3}=(\partial_{\pi}+\bar{\partial})(\bar{w_2})
\in\ker\, \partial_{\pi}\cap \ker\, \bar{\partial}\cap \mathrm{im} \,(\partial_{\pi}+\bar{\partial})
\neq\emptyset.
$$
Therefore $\mathbb{I}_3$ does not satisfy the $\partial_{\pi}\bar{\partial}$--lemma.

\subsection{A six--dimensional complex nilmanifold}

Motivated by the construction of Iwasawa manifold,
we consider complex nilpotent Lie group
$$
G=\Bigg\{A= \small{\left(\begin{array}{cccc}
1 &z_1 & z_{2} & z_{3}\\
0 &1 & z_{4} & z_{5}\\
0 & 0 & 1  & z_{6} \\
0& 0 & 0 & 1
\end{array}
\right)}\mid z_{i}\in \mathbb{C} \Bigg\}\subset \mathrm{GL}(4; \mathbb{C}).
$$
As complex manifolds, $G$ is isomorphic to $\mathbb{C}^{6}$.
Consider the discrete group
$H:=\mathrm{Gl}(4; \mathbb{Z}[i])\cap G$.
Analogously, the left multiplication gives a natural $H$--action on $G$.
The associated $H$--quotient space
$
\mathbb{I}_6:=G/ H
$
is a 6--dimensional compact complex manifold.
Moreover, a basis of the space $(\mathfrak{g}_{\mathbb{C}}^*)^{1,0}$ of
left--invariant holomorphic differential forms of $G$ is
given by:
\begin{eqnarray*}
 &&w_1 =dz_1,\;\; w_2=dz_2-z_1dz_4,\;\; w_3=dz_3-z_1dz_5+(z_1z_4-z_2)dz_6,  \\
&& w_4 =dz_4,\;\; w_5=dz_5-z_4dz_6,\;\; w_6=dz_6 \qquad \qquad\qquad
\end{eqnarray*}
with structure equations
$$
\begin{cases}
   dw_1=dw_4=dw_6=0,   &  \\
   dw_2=-w_1\wedge w_4,  &  \\
    dw_3=-w_1\wedge w_5-w_2\wedge w_6,     &  \\
     dw_5=-w_4\wedge w_6 .      &
\end{cases}
$$
Dually, the Lie algebra $\mathfrak{g}_{\mathbb{C}}^{1,0}$ of left--invariant holomorphic vector fields of $G$
has a basis
\begin{eqnarray*}
&&X_{1}=\frac{\partial}{\partial z_1},
 \;\; X_2=\frac{\partial}{\partial z_2},
 \;\; X_3=\frac{\partial}{\partial z_3},\qquad \qquad\qquad \\
&&X_4=\frac{\partial}{\partial z_4}+z_1\frac{\partial}{\partial z_2},
 \;\; X_5=\frac{\partial}{\partial z_5}+z_1\frac{\partial}{\partial z_3},\;\;
 X_6=\frac{\partial}{\partial z_6}+z_2\frac{\partial}{\partial z_3}+z_4\frac{\partial}{\partial z_5} \end{eqnarray*}
 whose only non--trivial relations are $[X_1,X_4]=X_2, [X_1,X_5]=X_3=[X_2,X_6], [X_4,X_6]=X_5$.
Consequently, $\mathbb{I}_6$ is holomorphically parallelizable,
a $6$--dimensional non--K\"{a}hler Calabi--Yau manifold
since it also does not satisfy the $\partial_{}\bar{\partial}$--lemma:
$$-w_1\wedge\bar{w_1}\wedge \bar{w_4}\notin\mathrm{im}\partial_{}\bar{\partial}$$
 while
$$
-w_1\wedge\bar{w_1}\wedge \bar{w_4}
=(\partial_{}+\bar{\partial})(w_1\wedge\bar{w_2})
\in\ker\, \partial_{}\cap \ker\, \bar{\partial}\cap \mathrm{im} \,(\partial_{}+\bar{\partial}).
$$
The Hodge diamond of $\mathbb{I}_6$ states as(see \cite[Appendix A]{CCYY22}):
\begin{center}
\footnotesize{\begin{tikzpicture}[scale=0.08]
\draw  (30,30) node[left] {$1$};
\draw  (25,25) node[left] {$6$};
\draw  (35,25) node[left] {$3$};
\draw  (20,20) node[left] {$15$};
\draw  (30,20) node[left] {$18$};
\draw  (40,20) node[left] {$5$};
\draw  (15,15) node[left] {$20$};
\draw  (25,15) node[left] {$45$};
\draw  (35,15) node[left] {$30$};
\draw  (45,15) node[left] {$6$};
\draw  (10,10) node[left]  {$15$};
\draw  (20,10) node[left]  {$60$};
\draw  (30,10) node[left]  {$75$};
\draw  (40,10) node[left]  {$36$};
\draw  (50,10) node[left]  {$5$};
\draw  (5,5) node[left] {$6$};
\draw  (15,5) node[left]  {$45$};
\draw  (25,5) node[left]  {$100$};
\draw  (35,5) node[left] {$90$};
\draw  (45,5) node[left]  {$30$};
\draw  (55,5) node[left]  {$3$};
\draw  (0,0) node[left] {$1$};
\draw  (10,0) node[left]  {$18$};
\draw  (20,0) node[left]  {$75$};
\draw  (30,0) node[left]  {$120$};
\draw  (40,0) node[left]  {$75$};
\draw  (50,0) node[left]  {$18$};
\draw  (60,0) node[left] {$1$};
\draw  (5,-5) node[left] {$3$};
\draw  (15,-5) node[left]  {$30$};
\draw  (25,-5) node[left]  {$90$};
\draw  (35,-5) node[left] {$100$};
\draw  (45,-5) node[left]  {$45$};
\draw  (55,-5) node[left]  {$6$};
\draw  (10,-10) node[left]  {$5$};
\draw  (20,-10) node[left]  {$36$};
\draw  (30,-10) node[left]  {$75$};
\draw  (40,-10) node[left]  {$60$};
\draw  (50,-10) node[left]  {$15$};
\draw  (15,-15) node[left] {$6$};
\draw  (25,-15) node[left] {$30$};
\draw  (35,-15) node[left] {$45$};
\draw  (45,-15) node[left] {$20$};
\draw  (20,-20) node[left] {$5$};
\draw  (30,-20) node[left] {$18$};
\draw  (40,-20) node[left] {$15$};
\draw  (25,-25) node[left] {$3$};
\draw  (35,-25) node[left] {$6$};
\draw  (30,-30) node[left] {$1$};
\end{tikzpicture}}
\end{center}

Here we only consider some special holomorphic Poisson structures on
$\mathbb{I}_6$ given by $G$--left--invariant holomorphic bi--vector fields:
$$
\pi_{1}=X_2\wedge X_3,\;\;\;
\pi_{2}=X_1\wedge X_3.
$$
Akin to the Iwasawa manifold, the holomorphic Koszul--Brylinski homology of
$\mathbb{I}_6$ can be computed in terms of the total cohomology of the double complex
$(\wedge^{\bullet,\bullet}\mathfrak{g}^{\ast}_{\mathbb{C}},\partial_\pi,\bar{\partial})$.
For the simplicity, we write
$w^{i_{1}\cdots i_{p}\bar{j_{1}}\cdots \bar{j_{q}}}
=w^{i_{1}}\wedge \cdots \wedge w^{i_{p}}\wedge w^{\bar{j_{1}}}\wedge\cdots\wedge w^{\bar{j_{q}}}$,
for any $1\leq p,q\leq6$.

\subsubsection{$\partial_{\pi}\bar{\partial}$--lemma on $(\mathbb{I}_{6},\pi_1)$}
Consider the holomorphic Poisson manifold $(\mathbb{I}_{6},\pi_1)$.
The only possible candidates of non--closed elements are
\begin{eqnarray*}
     \partial_{\pi_1}w^{23i_{1}\cdots i_{p-2}\bar{j_{1}}\cdots \bar{j_{q}}}&=&
     (\iota_{\pi_1}\circ \partial-\partial\circ \iota_{\pi_1})w^{23i_{1}\cdots i_{p-2}\bar{j_{1}}\cdots \bar{j_{q}}}\\
     &=&\iota_{\pi_1}(w^{23}\wedge \partial w^{i_{1}\cdots i_{p-2}\bar{j_{1}}\cdots \bar{j_{q}}})
     -\partial w^{i_{1}\cdots i_{p-2}\bar{j_{1}}\cdots \bar{j_{q}}}\\
          &=&\partial w^{i_{1}\cdots i_{p-2}\bar{j_{1}}\cdots \bar{j_{q}}}
     -\partial w^{i_{1}\cdots i_{p-2}\bar{j_{1}}\cdots \bar{j_{q}}}=0.
\end{eqnarray*}
This means the Dolbeault--Koszul--Brylinski spectral sequence of $(M,\pi_1)$ degenerates at $E_1$--page.
Analogous to $\mathbb{I}_3$,
$(\mathbb{I}_6,\pi_1)$ does not satisfy the $\partial_{\pi}\bar{\partial}$--lemma since
$\mathrm{im}\partial_{\pi}\bar{\partial}=\emptyset$, while
$$
-\bar{w_1}\wedge \bar{w_4}=(\partial_{\pi}+\bar{\partial})(\bar{w_2})
\in\ker\, \partial_{\pi}\cap \ker\, \bar{\partial}\cap \mathrm{im} \,(\partial_{\pi}+\bar{\partial})
\neq\emptyset.
$$

\subsubsection{$\partial_{\pi}\bar{\partial}$--lemma on $(\mathbb{I}_{6},\pi_2)$}
Consider the holomorphic Poisson manifold $(\mathbb{I}_{6},\pi_2)$.
Note that $\mathfrak{g}^{6,0}=\langle w^{123456}\rangle$,
  $\partial_{\pi_2} \mathfrak{g}^{6,0}=0$,  $\bar{\partial} \mathfrak{g}^{6,0}=0$,
thus we have
  \begin{center}
  $H_{0}(X, {\pi_2})=\langle [w^{123456}]\rangle=\mathbb{C}.$
  \end{center}
Furthermore, observe that $\mathfrak{g}^{5,0}
=\langle w^{23456},w^{13456},w^{12456},w^{12356},w^{12346},w^{12345}\rangle$,
$\bar{\partial} \mathfrak{g}^{5,0}=0$,
and the only non--closed monomial of $\mathfrak{g}^{5,0}$ under $\partial_{\pi_{2}}$ is
$
\partial_{\pi_{2}} w^{12356}
         =-w^{1456}.
$
Meanwhile, since $\mathfrak{g}^{6,1}
=\langle w^{123456\bar{1}},w^{123456\bar{2}},w^{123456\bar{3}},
w^{123456\bar{4}},w^{123456\bar{5}},w^{123456\bar{6}}\rangle$,
$\partial_{\pi_2} \mathfrak{g}^{6,1}=0$,
and the only non--closed monomials of $\mathfrak{g}^{6,1}$ under $\bar{\partial}$ are:
$$\bar{\partial}w^{123456\bar{2}}=-w^{123456\bar{1}\bar{4}}, \quad
    \bar{\partial}w^{123456\bar{3}}=-w^{123456\bar{1}\bar{5}}-w^{\bar{2}\bar{6}},\quad
     \bar{\partial}w^{123456\bar{5}}=-w^{123456\bar{4}\bar{6}},
$$
Consequently, we have that
\begin{multline*}
  H_{1}(\mathbb{I}_{6},{\pi_2})= \\
  \langle [w^{23456}],[w^{13456}],[w^{12456}],[w^{12346}],[w^{12345}],
[w^{123456\bar{1}}],[w^{123456\bar{4}}],[w^{123456\bar{6}}]\rangle=\mathbb{C}^{8}.
\end{multline*}
But $6+3>8$,
by the Lemma \ref{deg},
this means the Dolbeault--Koszul--Brylinski spectral sequence of $(\mathbb{I}_6,\pi_2)$
does not degenerate at $E_1$--page.

\subsubsection{Two submanifolds of $(\mathbb{I}_{6},\pi_2)$}
Let
$$\Gamma_1=
\Bigg\{\small{\left(\begin{array}{cccc}
1 &z_{1} & z_{2} & z_{3}\\
0 &1 & z_{4} & a_{24}\\
0 & 0 & 1  & a_{34} \\
0& 0 & 0 & 1
\end{array}
\right)}\mid z_{i}\in \mathbb{C},a_{24},a_{34}\in\mathbb{Z}[i] \Bigg\}.
$$
Naturally $(Y_1=\Gamma_1/H,\pi_2|_{Y_1}=X_1\wedge X_3)$ is a 4--dimensional nilmanifold,
a closed holomorphic Poisson submanifold of $\mathbb{I}_{6}$.
With the same arguments, one can check
that the Dolbeault--Koszul--Brylinski spectral sequence of  $(Y_1,\pi_3|_{Y_1})$  do not degenerate at $E_1$ page.

Take
$$\Gamma_2=
\Big\{A= \small{\left(\begin{array}{cccc}
1 &z_{1} & z_{2} & z_{3}\\
0 &1 & a_{23} & a_{24}\\
0 & 0 & 1  & a_{34} \\
0& 0 & 0 & 1
\end{array}
\right)}\mid z_{i}\in \mathbb{C},a_{23},a_{24},a_{34}\in\mathbb{Z}[i] \Big\}.
$$
Naturally $(Y_2=\Gamma_2/H,\pi_2|_{Y_2}=X_1\wedge X_3)$ is a complex 3--torus,
a closed holomorphic Poisson submanifold of $\mathbb{I}_{6}$.
Hence by Corollary \ref{Kahler},
the Dolbeault--Koszul--Brylinski spectral sequence of $(Y_2,\pi_3|_{Y_2})$
degenerates at $E_1$ page
while the one of $(\mathbb{I}_{6}, \pi_2)$ does not.

\subsection{Nakamura manifold}\label{Nak}

In this subsection, based on \cite[Example 3.4]{AK17}, we consider a special solvmanifold:
the (holomorphically parallelizable) Nakamura manifold.
Consider the complex Lie group
$G:=\mathbb{C}\ltimes_\phi\mathbb{C}^2$
where $\phi(z)=
\left(\begin{array}{cccc}
e^z &0 \\
0 &e^{-z}
\end{array}
\right)\in SL(2,\mathbb{C}).
$
There exist $a+\sqrt{-1}b\in \mathbb{C}$ and $c+\sqrt{-1}d\in \mathbb{C}$
such that $\mathbb{Z}(a+\sqrt{-1}b)+\mathbb{Z}(c+\sqrt{-1}d)$
is a lattice in $\mathbb{C}$ and $\phi(a+\sqrt{-1}b)$ and $\phi(a+\sqrt{-1}b)$
are conjugate to elements of $SL(4,\mathbb{Z})$,
where we regard $SL(2,\mathbb{C})\subset SL(4,\mathbb{R})$. Hence
there exists a lattice
$\Gamma=(\mathbb{Z}(a+\sqrt{-1}b)+\mathbb{Z}(c+\sqrt{-1}d))\ltimes_\phi\Gamma_{\mathbb{C}^2}$
of $G$, where $\Gamma_{\mathbb{C}^2}$ is a lattice of $\mathbb{C}^2$.
The holomorphically parallelizable solvmanifold $M=G/\Gamma$ is called
(holomorphically parallelizable) Nakamura manifold.

Denote by $\mathfrak{g}_{\mathbb{C}}$ the complexification
of the Lie algebra $\mathfrak{g}$ of $G$.
Once taking the coordinate $(z_1,z_2,z_3)$ of $G$,
we obtain a basis
$\{
X_1=\frac{\partial}{\partial z_1},
X_2=e^{z_1}\frac{\partial}{\partial z_2},
X_3=e^{-z_1}\frac{\partial}{\partial z_3}
\}$
of $(\mathfrak{g}_{\mathbb{C}})^{1,0}$ with Lie bracket
$$
[X_1,X_2]=X_2,
[X_1,X_3]=-X_3,
[X_2,X_3]=0.
$$
Meanwhile, the dual basis of the space $(\mathfrak{g}_{\mathbb{C}}^*)^{1,0}$ is
$\{
w_1=dz_1,
w_2=e^{-z_1}dz_2,
w_3=e^{z_1}dz_3
\}$
with the structure equations
$$
dw_1=0, dw_2=-w_1\wedge w_2, dw_3=w_1\wedge w_3.
$$
Therefore the $G$--left--invariant holomorphic bi--vector fields
$\pi_{12}=X_1\wedge X_2$ and $\pi_{23}=X_2\wedge X_3$
are Poisson.
Moreover, there exists a finite--dimensional subcomplex $B_\Gamma^{\bullet}$
of $(A^{0,\bullet}_M,\bar{\partial})$ such that
the inclusions $\iota: B_\Gamma^{\bullet}\hookrightarrow (A^{0,\bullet}_M,\bar{\partial})$
and
$$\iota: \wedge^{\bullet}(\mathfrak{g}_{\mathbb{C}}^*)^{1,0}\otimes B_\Gamma^{\bullet}
\hookrightarrow (A^{\bullet,\bullet}_M,\bar{\partial})$$
are quasi--isomorphisms.
Analogous to the Lemma \ref{nil},
we have that
the holomorphic Koszul--Brylinski homology of $(M,\pi)$ can be
computed in terms of the total cohomology of the double complex
$(\wedge^{\bullet}(\mathfrak{g}_{\mathbb{C}}^*)^{1,0}\otimes B_\Gamma^{\bullet},
\partial_{\pi}, \bar{\partial})$ if $\pi$ is $G$--left--invariant.
However, the Dolbeault cohomology of $M$ depends on the lattices $\Gamma$.
More precisely, the subcomplex $B_\Gamma^{\bullet}$ depends on two distinguished cases:
\begin{itemize}
\item[(1)]Both $b\in \mathbb{Z}\pi$ and $d\in \mathbb{Z}\pi$.
In this case,
$
B_\Gamma^{\bullet}=\wedge^\bullet \mathbb{C}\langle d\bar{z_1},
e^{-z_1}d\bar{z_2},e^{z_1}d\bar{z_3}\rangle
$ ;
\item[(2)]Either $b\notin \mathbb{Z}\pi$ or $d\notin \mathbb{Z}\pi$.
In this case,
\begin{eqnarray*}
  B_\Gamma^{1} &=& \mathbb{C}\langle d\bar{z_1}\rangle, \\
  B_\Gamma^{2} &=& \mathbb{C}\langle d\bar{z_2}\wedge d\bar{z_3}\rangle, \\
  B_\Gamma^{3} &=& \mathbb{C}\langle d\bar{z_1}\wedge d\bar{z_2}\wedge d\bar{z_3}\rangle.
\end{eqnarray*}
\end{itemize}
Note for both two cases, the Dolbeault operator $\bar{\partial}$ on $B_\Gamma^{\bullet}$
is trivial.
Therefore we have that the Nakamura manifold $M$ always do not satisfy the $\partial_{}\bar{\partial}$--lemma
since $\mathrm{im}\partial_{}\bar{\partial}=\emptyset$, while
$$
-w_1\wedge w_2\wedge \bar{w_1}
=(\partial_{}+\bar{\partial})(w_2\wedge\bar{w_1})
\in\ker\, \partial_{}\cap \ker\, \bar{\partial}\cap \mathrm{im} \,(\partial_{}+\bar{\partial}).
$$
Furthermore, if we consider $(M, \pi_{12}=X_1\wedge X_2)$,
one can check that $(M, \pi_{12})$ do not satisfy the $\partial_{\pi}\bar{\partial}$--lemma since
$\mathrm{im}\partial_{\pi_{12}}\bar{\partial}=\emptyset$, while
$$
-1=(\partial_{\pi_{12}}+\bar{\partial})(w_2)
\in\ker\, \partial_{\pi}\cap \ker\, \bar{\partial}\cap \mathrm{im} \,(\partial_{\pi}+\bar{\partial})
\neq\emptyset.
$$
However, if we consider $(M, \pi_{23}=X_2\wedge X_3)$,
then by the direct calculations, we have that on
$\wedge^{\bullet}(\mathfrak{g}_{\mathbb{C}}^*)^{1,0}\otimes B_\Gamma^{\bullet}$,
both differentials
$\partial_{\pi}$ and $\bar{\partial}$ are trivial,
hence $(M, \pi_{23})$ satisfies the $\partial_{\pi}\bar{\partial}$--lemma.



\begin{thebibliography}{GP}

\bibitem{Angella13}
D. Angella, {\it The cohomologies of the Iwasawa manifold and of Its small deformations}, J. Geom. Anal.
23 (2013) 1355--1378.

\bibitem{Angella}
D. Angella, {\it Cohomological aspects in complex non--K\"{a}hler geometry},
Lecture Notes in Mathematics, 2095. Springer, Cham, 2014.

\bibitem{AT13}
D. Angella, A. Tomassini,
{\it On the $\partial_{}\bar{\partial}$--lemma and Bott--Chern cohomology},
Invent. Math. 192 (2013), no. 1, 71--81.

\bibitem{AK17}
D. Angella, H. Kasuya, {\it Bott--chern cohomology of solvmanifolds},
Ann. Glob. Anal. Geom. 52(2017), 363--411.

\bibitem{AT15}
D. Angella, A. Tomassini,
{\it Inequalities \`{a} la Fr\"{o}licher and cohomological decompositions},
J. Noncommut. Geom. 9 (2015), no. 2, 505--542.

\bibitem{BK98}
S. Barannikov, K. Kontsevich,
{\it Frobenius manifolds and formality of Lie algebras of polyvector fields},
Internat. Math. Res. Notices 1998, no. 4, 201--215.

\bibitem{BCV19}
M. Bailey, G. R. Cavalcanti, and J. L. van der Leer Dur\'{a}n,
{\it Blow--ups in generalized complex geometry},
Trans. Amer. Math. Soc.  371 (2019),  2109--2131.


\bibitem{BX15}
D. Broka, P. Xu,
{\it Symplectic realizations of holomorphic Poisson manifolds},
to appear in Math. Res. Lett. arXiv:1512.08847.


\bibitem{Bry88}
J.-L. Brylinski,
{\it A differential complex for Poisson manifolds},
J. Differential Geom.  28 (1988), 93--114.


\bibitem{BZ99}
J.-L. Brylinski, G. Zuckerman,
{\it The outer derivation of a complex Poisson manifold},
J. Reine Angew. Math. 506 (1999), 181--189.


\bibitem{Ba13}
M. Bailey,
{\it Local classification of generalized complex structures},
J. Differential Geom. 95 (2013), 1--37.

\bibitem{Cav06}
G. R. Cavalcanti,
{\it The decomposition of forms and cohomology of generalized complex manifolds},
 J. Geom. Phys. 57 (2006), no. 1, 121--132.

\bibitem{CCYY22}
X.\ Chen, Y.\ Chen, S.\ Yang, X.\ Yang,
{\it Holomorphic Koszul--Brylinski homologies of Poisson blow--ups}, arXiv:2202.09764v2.

\bibitem{CFP16}
Z.\ Chen, A.\ Fino, Y.-S. Poon,
{\it Holomorphic Poisson structure and its cohomology on nilmanifolds},
Differential Geom. Appl. 44 (2016), 144--160.

\bibitem{CGP15}
Z.\ Chen, D.\ Grandini, Y.-S. Poon,
{\it Holomorphic Poisson cohomology},
Complex Manifolds 2 (2015), 34--52.


\bibitem{CSX10}
Z. Chen, M. Sti\'{e}non, P. Xu,
{\it Geometry of Maurer--Cartan elements on complex Manifolds},
Comm. Math. Phys. 297 (2010),  169--187.

\bibitem{CFGU00}
L.A. Cordero, M. Fernandez, A. Gray, L. Ugarte,
{\it Compact nilmanifolds with nilpotent complex structures: Dolbeault cohomology}, Trans. Am. Math. Soc. 352 (2000), 5405--5433.

\bibitem{DGMS75}
P. Deligne, P. Griffiths, J. Morgan, D. Sullivan,
{\it Real homotopy theory of K\"{a}hler manifolds},
Invent. Math. 29 (1975), 245--274.



\bibitem{FM12}
D. Fiorenza, M. Manetti,
{\it Formality of Koszul brackets and deformations of holomorphic Poisson manifolds}, Homology, Homotopy Appl. 14 (2012), 63--75.

\bibitem{Fu05}
B. Fu,
{\it Poisson resolutions},
J. reine angew. Math. 587  (2005), 17--26.


\bibitem{Go10}
R. Goto,
{\it Deformations of generalized complex and generalized K\"{a}hler structures},
J. Differential Geom. 84 (2010), 525--560.



\bibitem{Gu11}
M. Gualtieri,
{\it Generalized complex geometry},
 Ann. of Math.  174 (2011), 75--123.

\bibitem{Hit03}
N.J. Hitchin,
{\it Generalized Calabi--Yau manifolds},
Quart. J. Math. 54 (2003), 281--308.


\bibitem{Hit06}
N.J. Hitchin,
{\it Instantons, Poisson structures and generalized K\"{a}hler geometry},
Comm. Math. Phys. 265 (2006), 131--164.

\bibitem{Hit12}
N. Hitchin, {\it Deformations of holomorphic Poisson manifolds},
Mosc. Math. J. 12 (2012), no. 3, 567--591, 669.

\bibitem{Hon19}
W.\ Hong,
{\it Poisson cohomology of holomorphic toric Poisson manifolds. I},
J. Algebra, 527 (2019), 147--181.

\bibitem{HX11}
W.\ Hong, P.\ Xu,
{\it Poisson cohomology of Del Pezzo surfaces},
J. Algebra, 336 (2011), 378--390.

\bibitem{Kos84}
J. L. Koszul,
{\it Crochet de Schouten--Nijenhuis et cohomologie}, The mathematical heritage of \'{E}lie Cartan (Lyon, 1984).
Ast\'{e}risque 1985, Num\'{e}ro Hors S\'{e}rie, 257--271.

\bibitem{L-GSX08}
C.\ Laurent-Gengoux, M.\ Sti\'{e}non, P.\ Xu,
{\it Holomorphic Poisson manifolds and holomorphic Lie algebroids},
Int.  Math. Res. Not. 2008, Art. ID rnn 088, 46 pp.


\bibitem{Merkulov98}
S. A. Merkulov,
{\it Formality of canonical symplectic complexes and Frobenius manifolds},
Internat. Math. Res. Notices, 1998, no. 14, 727--733.

\bibitem{Man99}
 Y. I. Manin,
{\it Frobenius Manifolds, Quantum Cohomology, and Moduli Spaces},
Amer. Math. Soc. Colloq. Publ., vol. 47. Providence, RI: American Mathematical Society, 1999.




\bibitem{Sak76}
Y.\ Sakane, {\it On compact complex parallelisable solvmanifolds},
Osaka J. Math.
13 (1976), 187--212.

\bibitem{ST08}
G. Sharygin, D. Talalaev,
{\it On the Lie--formality of Poisson manifolds},
J. K-Theory, 2 (2008), no. 2, Special issue in memory of Yurii Petrovich
Solovyev, Part 1, 361--384.

\bibitem{Sti11}
M. Sti\'{e}non,
{\it Holomorphic Koszul--Brylinski Homology},
Int.  Math. Res. Not. 2011 (2011),  553--571.

\bibitem{TY12-1}
L.-S. Tseng, S.-T. Yau,
{\it Cohomology and Hodge Theory on Symplectic Manifolds: I},
J. Differential Geom. 91 (2012), no. 3, 383--416.

\bibitem{TY12-2}
L.-S. Tseng, S.-T. Yau,
{\it Cohomology and Hodge Theory on Symplectic Manifolds: II},
J. Differential Geom. 91 (2012), no. 3, 417--443.

\bibitem{TY16}
L.-S. Tseng, S.-T. Yau,
{\it Cohomology and Hodge theory on symplectic manifolds: III},
 J. Differential Geom. 103 (2016), no. 1, 83--143.

\bibitem{We97}
A. Weinstein,
{\it The modular automorphism group of a Poisson manifold},
J. Geom.  Phys.  23 (1997), 379--394.

\bibitem{YY20}
S.\ Yang, X.\ Yang,
{\it Bott-Chern blow--up formulae and the bimeromorphic invariance of the $\partial\overline{\partial}$--lemma for threefolds},
Trans. Amer. Math. Soc.  373 (2020), 8885--8909.


\end{thebibliography}
\end{document}